\documentclass{amsart}
\usepackage{amsmath,amssymb,dsfont,amsthm}
\usepackage{graphics}
\usepackage{xypic}
\usepackage{latexsym}
\usepackage{amsmath}
\usepackage{amsfonts}
\usepackage{geometry}
\usepackage{amssymb}
\usepackage{tensor}

\date{\today}

\title[On the Cheeger-M\"{u}ller theorem for a cone]{On  the Cheeger-M\"{u}ller theorem for an even dimensional cone }


\author{L. Hartmann}
\address[Luiz Hartmann]{\tt UFSCar, Universidade Federal de S\~{a}o Carlos, S\~{a}o Carlos, Brazil. \newline Partially supported by CNPq and FAPESP 2013/04396-6}
\email{hartmann@dm.ufscar.br}

\author{ M. Spreafico}
\address[Mauro Spreafico]{\tt Universit\`a del Salento, Lecce, Italy.} 
\email{mauro.spreafico@unisalento.it}

\numberwithin{equation}{section}

\newtheorem{theo}{Theorem}[section]
\newtheorem{lem}{Lemma}[section]

\newtheorem{prop}{Proposition}[section]

\renewcommand{\S}{\mathcal{S}}

\renewcommand{\Re}{{\rm Re}}
\renewcommand{\Im}{{\rm Im}}

\newcommand{\Sp}{{\rm Sp}}
\newcommand{\beq}{\begin{equation}}
\newcommand{\eeq}{\end{equation}}

\newcommand{\Z}{{\mathds{Z}}}
\newcommand{\R}{{\mathds{R}}}

\newcommand{\Q}{{\mathds{Q}}}
\newcommand{\T}{{\mathcal{T}}}
\renewcommand{\t}{{\mathcal{t}}}

\newcommand{\A}{{\mathcal{A}}}
\renewcommand{\H}{{\mathcal{H}}}

\renewcommand{\P}{{\mathcal{P}}}
\newcommand{\B}{{\mathcal{B}}}

\newcommand{\Aut}{{\rm Aut}}

\renewcommand{\b}{{\partial}}

\newcommand{\p}{{\mathsf p}}
\renewcommand{\t}{{\mathsf t}}
\newcommand{\m}{{\mathsf m}}

\date{}

\DeclareMathOperator*{\Rz}{Res_0}

\DeclareMathOperator*{\Ru}{Res_1}

\begin{document}


\maketitle

\begin{abstract} We prove the equality of the $L^2$-analytic torsion and the intersection R torsion of the even dimensional finite metric cone over an odd dimensional compact manifold. 
\end{abstract}

\section{Introduction}

The  classical Cheeger-M\"{u}ller theorem proves equality between $L^2$-analytic torsion and Reidemeister R torsion for closed Riemannian manifolds  \cite{Che1,Mul,RS}. When the manifold has a boundary, a boundary term appears. This boundary term (given by L\"{u}ck in \cite{Luc} when the metric is a product near the boundary) has been explicitly given in the general case in some recent works of J. Br\"uning and X. Ma \cite{BM1,BM2}. For a compact connected oriented Riemannian manifold $(M,g_M)$ with boundary the Cheeger-M\"uller theorem reads (a similar formula is valid for the relative case)
\begin{align*}
\log T_{\rm abs}((M,g_M);\rho)&=\log \tau_{\rm R}((M,g_M);\rho)+\frac{1}{4}{\rm rk}(\rho)\chi(\b M)\log 2+{\rm rk}(\rho)A_{\rm BM, abs}(\b M),
\end{align*}
where $T_{\rm abs}((M,g_M);\rho)$ and $\tau_{\rm R}((M,g_M);\rho)$ are the analytic torsion with absolute BC on the boundary, and the Reidemeister R torsion of $(M,g_M)$, with respect to an orthogonal representation $\rho$ of the fundamental group of $M$ and with the basis for homology fixed as in \cite{RS} (see Sections \ref{dr} and \ref{tor} for details), respectively, $\chi$ is the Euler characteristic, and $A_{\rm BM}$ is the anomaly boundary term. 

In this work we prove the following extension of the Cheeger-M\"uller theorem, where $M$ is the cone $C W$   over a compact connected oriented Riemannian manifold $(W,g)$ (see Section \ref{cone} for details) and where the usual R torsion is replaced by the intersection torsion $I\tau_R$ defined by A. Dar in \cite{Dar1} (where the homology basis is fixed using the $L^2$-harmonic forms via the suitable De Rham map, see Section \ref{tor} for details). The proof follows at once from the formula for the analytic torsion given in Theorem \ref{t01} in Section \ref{cone}, and the duality formula for the intersection torsion proved in Proposition \ref{p1} in the last section. Since the metric on $CW$ is fixed by the metric of $W$, and the unique representation of the fundamental group is the trivial one, that may be assumed of rank one, both these quantities will be omitted in the notation. 

\begin{theo}\label{t1} Let $(W,g)$ be a compact connected oriented Riemannian odd dimensional manifold  without boundary. Let $CW$ denote the cone over $W$. Then,
\begin{align*}
\log T_{\rm abs}(C W)&=\log I \tau_{\rm R}(C W)+A_{\rm BM, abs}(W),\\
\log T_{\rm rel}(C W)&=\log I \tau_{\rm R}(C W,\b C W)+A_{\rm BM, rel}(W).
\end{align*}

\end{theo}

We conclude this introduction with some remarks.

\begin{itemize}

\item If the dimension of $W$ is even,  results for the analytic torsion (see \cite{HS4,MV}) exist and some extra terms appear, whose interpretation is still not clear. For this reason we omit not illuminating formulas and we concentrate here on the odd case. There is work in progress in this direction. 

\item The result of this work will give the extension of the Cheeger-M\"uller theorem for a general space with conical singularities (as define in \cite{Che2}) once it will be available a suitable gluing formula extending the one proved by S.M. Vishik  in \cite{Vis} for compact manifolds under the assumption that  the metric is product near the gluing. This is a likely result, for  recently on one side a gluing formula for compact manifolds without any assumption on the metric near the boundary was given by  J. Br\"uning and X. Ma in \cite{BM2}, and on the other side M. Lesch \cite{Les} extended the result of S.M. Vishik to a pseudomanifold. 
In particular, since in general pseudomanifold are modeled on cones, a generalization of the result of this work for a cone over a pseudomanifold could be used to obtain a generalization of the Cheeger-M\"uller theorem for pseudomanifolds. 

\item The result of Theorem \ref{t1} could be read in terms of Euler basis \cite{FT} (and in term of the Reidemeister  metric \cite{BZ,Qui} ), and would affirm that the Euler basis of $C_l W$ coincides with  the quotient of the De Rham basis  by the $L^2$-analytic torsion (in term of metrics, the intersection Reidemeister metric coincides with the Ray and Singeer metric, namely  the quotient of the De Rham  metric by the $L^2$-analytic torsion). However, while  the analytic side is clear, the complete development of a precise  theory  for the intersection Reidemeister basis (metric) is still under construction (see Section 5 of \cite{HS3}). There is work in progress on this topic. 

\end{itemize}

\section{Background and analytic torsion}
\label{s2}

This section is essentially based on \cite{HS1} and \cite{HS2}, and we refer to those papers for further details (see also \cite{Ver2} for the analytic torsion).

\subsection{Geometric setting}
\label{s0}

Let $(M,g_M)$ be a compact connected oriented Riemannian manifold of dimension $m$ with boundary $\b M$ and Riemannian structure $g_M$. Let $\rho:\pi_1(M)\to O(k,\R)$ be a representation of the fundamental group of $M$, and let  $E_\rho$ be the associated vector bundle over $M$ with fibre $\R^k$ and group $O(k,\R)$, $E_\rho=\R^k \times_\rho \widetilde M$. Let $\Omega(M,E_\rho)$ denote the graded linear space of smooth forms on $M$ with values in $E_\rho$.  The exterior differential on $M$ defines the exterior differential on $\Omega^q(M, E_\rho)$, $d:\Omega^q(M, E_\rho)\to
\Omega^{q+1}(M, E_\rho)$. The metric $g$ defines an Hodge operator on $M$ and hence on $\Omega^q(M, E_\rho)$, $\star:\Omega^q(M, E_\rho)\to\Omega^{m-q}(M, E_\rho)$, and,  using the inner product $\langle\_,\_\rangle$ in $\R^k$,
an inner product on $\Omega^q(M, E_\rho)$ is defined by
\[
(\omega,\eta)=\int_M \langle \omega\wedge\star\eta \rangle.
\]

Near the boundary there is a natural splitting of $\Lambda M$ as direct sum of vector bundles $\Lambda T^*\b M\oplus N^* M\otimes \Lambda T^*\b M$, where $N^*M$ is the dual to the normal bundle to the boundary, and the smooth forms on $M$ near the boundary decompose as $\omega=\omega_{\rm tan}+\omega_{\rm norm}$, where $\omega_{\rm norm}$ is the orthogonal projection on the subspace generated by $dx$, the one form corresponding to the outward pointing unit normal vector to the boundary, and $\omega_{\rm tan}$ is in $C^\infty(M)\otimes\Lambda(\b M)$. We  write $\omega=\omega_1+ dx \wedge\omega_{2}$, where $\omega_j\in C^\infty( M, \Lambda(T^*\b M))$, and
\[
\star\omega_2=dx \wedge \star\omega.
\]

Define absolute and relative boundary conditions by
\[
B_{\rm abs}(\omega)=\omega_{\rm norm}|_{\b M}=\omega_2|_{\b M}=0,\qquad 
B_{\rm rel}(\omega)=\omega_{\rm tan}|_{\b M}=\omega_1|_{\b M}=0.
\]

Let $\B(\omega)=B(\omega)\oplus B((d+d^\dagger)(\omega))$. The adjoint $d^\dagger$ and the Laplacian $\Delta=(d+d^\dagger)^2$ operators are defined on the space of sections with values in $E_\rho$,  the Laplacian with boundary conditions $\B(\omega)=0$  is self adjoint, and the spaces  of the harmonic forms with boundary conditions are
\begin{align*}
\H_{\rm abs}^q(M,E_\rho)&=\{\omega\in\Omega^q(M,E_\rho)~|~\Delta^{(q)}\omega=0, \B_{\rm abs}(\omega)=0\},\\
\H_{\rm rel}^q(M,E_\rho)&=\{\omega\in\Omega^q(M,E_\rho)~|~\Delta^{(q)}\omega=0, \B_{\rm rel}(\omega)=0\}.
\end{align*}


\subsection{De Rham maps}\label{dr} Let $K$ be a cellular or simplicial decomposition of $M$ and $L$ of $\b M$. Let $C_q(M;E_\rho)=\R^k\otimes_\rho C_q(\tilde M;\Z\pi_1(M))$ be complex of the twisted chains (see the paragraph before equation (\ref{ppp}) for details on this construction). Then we have the following de Rham maps $\A^q$ (that induce isomorphisms in cohomology),
\begin{align*}
\A_{\rm abs}^q:&\H^q_{\rm abs}(M,E_\rho)\to C^q(M;E_\rho),&
\A_{\rm rel}^q:&\H_{\rm rel}^q(M ,E_\rho)\to C^q((M,\b M);E_\rho),
\end{align*}
with
\[
\A_{\rm abs}^q(\omega)(v\otimes_\rho c)=\A^q_{\rm rel}(\omega)(v\otimes_\rho c)=\int_{ c} (\omega,v),
\]
where $v\otimes_\rho c$ belongs to $C_q(M;E_\rho)$
, and $c$ is identified with the $q$-subcomplex (simplicial or cellular) that $c$ represents. 
Following Ray and Singer \cite{RS}, we introduce the de Rham maps $\A_q$:
\begin{align*}
\A^{\rm rel}_q:&\H_{\rm rel}^q(M,E_\rho)\to C_q((M,\b M);E_\rho),&
\A^{\rm rel}_q:&\omega\mapsto (-1)^{(m-1)q}\P_q^{-1}\A_{\rm abs}^{m-q}\star(\omega),\\
\A^{\rm abs}_q:&\H_{\rm abs}^q(M,E_\rho)\to C_q(M;E_\rho),&
\A^{\rm abs}_q:&\omega\mapsto (-1)^{(m-1)q}\P_q^{-1}\A_{\rm rel}^{m-q}\star(\omega),
\end{align*}
both defined by
\beq\label{aa}
\A^{\rm rel}_q(\omega)=\A^{\rm abs}_q(\omega)=(-1)^{(m-1)q}\sum_{j,i} \left(\int_{\hat c_{q,j}}(\star\omega,e_i)\right)
c_{q,j}\otimes_\rho e_i,
\eeq
where the sum runs over all $q$-simplices $c_{q,j}$ of $M-\b M$ in the first case, but runs over all $q$-simplices $c_{q,j}$ of $M$ in the second case. Here $\P_q:C_q(K,L;\Z)\to C^{m-q}(\hat K-\hat L;\Z)$ is the Poincar\'e map, and $\hat c$ denotes the dual block cell of $c$. The extension of the de Rham map for pseudomanifolds is based on the works on Cheeger \cite{Che2}, see the end of Section \ref{tor} below.

\subsection{Zeta function and analytic torsion} The Laplace operator on forms $\Delta^{(q)}$, with boundary conditions $\B_{\rm abs/rel}$,  has a pure point spectrum $\Sp \Delta_{\rm abs/rel}^{(q)}$ consisting of real non negative eigenvalues. The sequence $\Sp_+ \Delta_{\rm abs/rel}^{(q)}$ is a totally regular sequence of spectral type accordingly to \cite{HS1} Section 4, and the  {\it forms valued zeta function} is the associated zeta function, defined by
\[
\zeta(s,\Delta_{\rm abs/rel}^{(q)} )=\zeta(s,\Sp_+ \Delta_{\rm abs/rel}^{(q)} )=\sum_{\lambda\in\Sp_+ \Delta_{\rm abs/rel}^{(q)}}\lambda^{-s},
\]
when $\Re(s)>\frac{m}{2}$, and by analytic continuation elsewhere. The {\it analytic torsion} $T_{\rm abs/rel}((W,g);\rho)$ of $(W,g)$ with respect to the representation $\rho$ is defined by 
\[
\log T_{\rm abs/rel}((W,g);\rho)=\frac{1}{2}\sum_{q=1}^m (-1)^q q \zeta'(0,\Delta_{\rm abs/rel}^{(q)}).
\]

\subsection{The analytic torsion of a cone} \label{cone}

Let $(W,g)$ be an orientable  compact connected Riemannian manifold of  dimension $m$ without boundary  and with Riemannian structure $g$. We denote by $C_lW$ the space $([0,l]\times W)/(\{0\}\times W)=(0,l]\times W\cup \{p\}$, where $p$ is the vertex of the cone, i.e. the image of $\{0\}\times W$ under the quotient map,  with the metric 
\[
g_C=dx\otimes dx+x^2 g,
\]
on $(0,l]\times W$, and we call it the {\it finite metric cone} over $W$ (see \cite{HS2} 3.1 for details). The analytic torsion of a cone over a sphere (i.e. $W=S^m$) was studied in \cite{HS1}. The result is based on one side on works of J. Cheeger on the Hodge theory of $L^2$ forms \cite{Che2,Che3,Nag}, and on the other on works of M. Spreafico on zeta invariants for double sequences \cite{Spr3,Spr4,Spr5,Spr9}. In the general case, extending the approach used for the spheres in \cite{HS1}, we have the following result:
 
\begin{theo}\label{t01} The analytic torsion of the cone $C_l W$ on an orientable compact connected Riemannian manifold $(W,g)$ of odd dimension $m=2p-1$ is
\begin{align*}
\log T_{\rm abs}(C_lW, g_C)= &\frac{1}{2} \sum_{q=0}^{p-1} (-1)^{q} {\rm rk}H_q(W;\Q)\log \frac{l}{2(p-q)}+\frac{1}{2} \log T(W,l^2g)+{\rm S}(\b C_l W),
\end{align*}
where the singular term ${\rm S}(\b C_l W)$ only depends on the boundary of the cone:
\[
{\rm S}(\b C_l W)=\frac{1}{2}\sum_{q=0}^{p-1}  \sum_{j=0}^{p-1}\sum^{j}_{k=0} \Rz_{s=0}\Phi_{2k+1,q}(s)
\binom{-\frac{1}{2}-k}{j-k} \sum^{q}_{h=0}(-1)^{h}\Ru_{s=j+\frac{1}{2}}\zeta\left(s,\tilde \Delta^{(h)}\right)(q-p+1)^{2(j-k)},
\]
(the functions $\Phi_{2k+1,q}(s)$ are some universal functions, explicitly known by some recursive relations, and $\tilde\Delta$ is the Laplace operator on forms on the section of the cone) and coincides with the anomaly boundary term of Br\"uning and Ma, namely ${\rm S}(\b C_l W)=A_{\rm BM,abs}(\b C_l W)$.
\end{theo}


The proof of Theorem \ref{t01} is based on analytic tools and is essentially the same as the proof of similar results for  the spheres given in \cite{HS1}. In the general case treated here, using the same method and a similar strategy, we just need to solve several technical problems, that can be quite hard, and require long difficult analysis.  All details are in \cite{HS2}. See also \cite{Ver2}.

\section{Intersection torsion}
\label{ss1}

Intersection torsion for pseudomanifolds was introduced in works of A. Dar \cite{Dar1,Dar2}. In these works the case of pseudomanifolds without boundary is considered, and in general all intersection homology theory is developed for 
the boundaryless case. Here we need to consider the boundary case, but a particular situation where the boundary is in fact a smooth manifold, disjoint from the singular locus. In this particular case it is easy  to rework all definitions and the main results of the boundaryless case, as expect. This is the purpose of this section.
 
\subsection{Pseudomanifolds with smooth boundary} We define pseudomanifolds with smooth boundary adapting the definition of pseudomanifolds of \cite{GM1} 1.1, \cite{Hae} 1, \cite{KW}  4.1. If $X$ is a topological space, we denote by $CX$ the cone over $X$. 
By definition the cone and the open cone over the empty set are  a point.
A {\it topological pseudomanifold of dimension $0$} is a countable set with the discrete topology. A {\it topological pseudomanifold of dimension $n$ with smooth boundary} is an Hausdorff paracompact topological space $X$ with a filtration by closed subspaces
\[
\emptyset=X_{-1}\subseteq X_0\subseteq X_1\subseteq \dots \subseteq X_{n-3}\subset X_{n-2}=X_{n-1}=\Sigma\subset X_n=X,
\]
called {\it stratification}, such that: (1) $X-\Sigma$ is dense in $X$;
(2) there exists a closed subspace $B$ of $X$, with $B\cap \Sigma=\emptyset$, such that $M=X-\Sigma$ is an $n$-manifold with boundary $\b M=B$, and for each $j\leq n-2$, 
for each point $x\in X_j-X_{j-1}$ there exists a compact topological pseudomanifold $L$ of dimension $n-j-1$  with filtration
\[
\emptyset=L_{-1}\subseteq L_0\subseteq L_1\subseteq\dots \subseteq L_{n-j-3}=L_{n-j-2}\subset L_{n-j-1}=L,
\]
and  a neighborhood $U_x$ of $x$ in $X$ with  an homeomorphism
$\varphi:U_x\to \R^j\times \mathring{C}L$,
which respects the stratifications, namely $\varphi$  maps homeomorphically 
$U_x\cap X_{j+k+1}$ onto $\R^j\times \mathring{C}L_k$.


The closed subspace $\Sigma=X_{n-2}$ is called the {\it singular locus} of $X$. It is a consequence of the definition that each subspace $X_j-X_{j-1}$ is a manifold of dimension $j$ with boundary, and, by condition (2), the boundary of $X$ is disjoint from the singular locus. When the singular locus has dimension 0, then a stratification of $X$ is 
\[
\emptyset=X_{-1}\subset \Sigma=X_0=X_1=\dots =X_{n-1}\subset X,
\]
and $X$ is called a {\it space with isolated singularities}. In this work we are mainly concerned with this type of pseudomanifolds. If $X$ is a manifold with boundary, then $X$ is a pseudomanifold with a stratification consisting with only one stratum $X$. For our purpose it is sufficient to work in the piecewise linear category, as in \cite{GM1}.
A {\it piecewise linear (pl) space} $X$ is a topological space with a class of locally finite simplicial triangulations $\T(X)$: if $T\in \T$ then any (linear) subdivision of $T$ belongs to $\T(X)$, and it $T_1,T_2\in \T(X)$, then they have a common subdivision in $\T(X)$. A closed pl-subspace of $X$ is a subspace which is a 
subcomplex of a suitable admissible triangulation of $X$.  We will identify a triangulation of a space with the associated simplicial complex. A {\it pl-pseudomanifold $X$ of dimension $n$ with smooth boundary} is a pl-space $X$ of 
dimension $n$ containing two closed disjoint  pl-subspaces $\b X$ and $\Sigma$, with $\Sigma$ of codimension greater or equal to $2$, such that $X-\Sigma$ is an oriented  pl-manifold of dimension $n$  dense in $X$ and with smooth boundary $\b X$. 
Equivalently, for an (admissible) triangulation of $X$, then $X$ is 
the union of the closed $n$-simplices and each $(n-1)$-simplex is face of 
one or  two $n$-simplices, and $\b X$ is the subcomplex of the $(n-1)$-simplices that are faces of just one $n$-simplex.
By the same proof as in  \cite{Hae} Prop. 1.4, any pl-pseudomanifold with smooth boundary admits a pl-stratification: a stratification of $X$ is given by setting $X_k=|T_{(k)}|$, where $T$ is an (admissible) triangulation of $X$, and this stratification is subordinate to the triangulation, meaning that the strata are subcomplexes. 

From now on we assume pseudomaniofolds are finite pl-pseudomanifolds, that pl-pseudomanifolds  have a  fixed stratification (the previous one if a triangulation is given), and that all triangulations are {\it admissible}, i.e. compatible with the pl-structure. Our definition of pseudomanifold with smooth boundary is consistent with the definition of pseudomanifold with boundary of \cite{GM1} 5.2, taking a manifold for boundary, namely assuming the singular locus of the boundary  vanishes.

\subsection{Intersection homology and relative intersection homology for pseudomanifolds with smooth boundary}

Let first recall the basic ingredients for the definition of intersection homology, as in  \cite{GM1}. A {\it perversity} is a finite sequence of integers $\p=\{\p_j\}_{j=2}^n$ such that $\p_2=0$ and $\p_{j+1}=\p_j$ or $\p_j+1$. The perversity: $\m=\{\m_j=[j/2]-1\}$ is called {\it lower middle perversity}. The {\it null perversity} is $0_j=0$, and the {\it top perversity} is $\t_j=j-2$. Given a perversity $\p$, the {\it complementary perversity} $\p^c$ is $\p^c_j=\t_j-\p_j=j-\p_j-2$. Now let $X$ be a pseudomanifold with boundary and with a given stratification. If $j$ is an integer and $\p$ a perversity, a pl-subspace $A$ of $X$ is said {\it $(\p,j)$-allowable} if
\[
\dim(A)\leq j, \hspace{20pt} \dim(A\cap X_{n-k})\leq j-k+p_k,\hspace{10pt} \forall k\geq 2.
\]

In standard references intersection homology is usually defined for  pseudomanifolds without boundary, and relative intersection homology for pairs $(X,A)$ where $A$ is an open subspace of a closed pseudomanifold. In order to extend the definition to pseudomanifolds with smooth boundary we have, at least, two possible equivalent approaches. The first approach is based on \cite{Mac}, and use smoothly enclosed subspaces, as follows.  Glue the infinite cylinder $\b X\times [0,\infty)$ to $X$ through the boundary, let $Z=X\cup_{\b X} \b X\times [0,\infty)$. Embed  $Z$ into the suitable $ \R^{k=k_1+k_2}$, in such a way that $i(\b X)=i(X)\cap \R^{k_1}\times \{0, \ldots, 0\}$,  $i(\Sigma)\subset \{x\in \R^k~|~x_j>0, \forall j>k_1\}=\R^{k_1}\times \R^{k_2}_+-\R^{k_1}\times \{0, \ldots, 0\}$,  and  $i(\b X\times [0,\infty))\subset \{x\in \R^k~|~x_j\leq 0, \forall j>k_1\}=\R^{k_1}\times \R^{k_2}_-$, where $i$ denotes the embedding. Then  a Whitney stratification of $\R^k$ is given by setting $Z_0=\R^k-i(Z)$, $Z_1=i(X-\Sigma)$, $Z_k=i(X_{k}-X_{k-1})$ for $k\geq 2$ (see \cite{Mac} Section 7.1.2), and $i(X)$ is the closure of $Z_1$. The   subsets $S_\pm=\R^{k_1}\times \R^{k_2}_\pm$ are smoothly enclosed in $\R^n$, as in the definition of Section 1.3.2 of \cite{Mac}
. Now (identifying the different spaces with their images under $i$)  it is clear that $S_+\cap Z=X$. By definition \cite{Mac} Section 1.2.3, $X$ is smoothly enclosed in $Z$, and $\b X=(S_-\cap Z)\cap (S_+\cap Z)=(S_-\cap Z)\cap X$ is the intersection with another smoothly enclosed subset of $Z$. It follows from \cite{Mac} Section 1.2.3 that  both the intersection chain complex $I^\p C(X)$ of $X$, and the relative intersection chain complex $I^\p C (X,\b X)$ of the pair $(X,\b X)$ are defined, the first as in the boundary less case,  the last one  by setting $I^\p C_q (X,\b X)=I^\p C_q (X)/I^\p C_q (\b X)$. The intersection homology groups are the homology groups of $I^p C(X)$, and  the relative intersection homology group $I^\p H_q(X,\b X)$ of the pair is defined as the $q$-homology group of the chain complex $I^\p C (X,\b X)$. Moreover,  there is the following homology long exact sequence associated to the pair $(X,\b X)$ (see also \cite{GM3} 1.11):
\[
\ldots \to I^\p H_q(\b X)\to I^\p H_q(X)\to I^\p H_q(X,\b X)\to I^\p H_{q-1} (\b X)\to \dots.
\]

The second approach  proceeds as in \cite{GM2} 1.4, and consists in replacing the pseudomanifold $X$ with boundary $\b X$  by the pseudomanifold $X-\b X$. For let $X$ be a pesudomanifold with smooth boundary $\b X$. Let $Col(\b X)$ be an open collar neighborhood of $\b x$. Then, $X-\b X$ is a pseudomanifold, with open subspace $Col(\b X)-\b X$, and usual intersection homology theory and relative intersection homology theory are defined for $X-\b X$, and $ (X-\b X,Col(\b X)-\b X)$,  \cite{GM2} 1.3 \cite{KW} 4.6. Since the boundary is disjoint from the singular stratum,  there exists a stratum preserving  homotopy self equivalence $X\sim X-\b X$, and the same for the pair $(X,\b X)\sim (X-\b X,Col(\b X)-\b X)$, as defined in \cite{KW} 4.8. It follows by \cite{KW} 4.8.5, that $I^\p H_q(X)\cong I^\p H_q(X-\b X)$, and $I^\p H_q(X,\b X)=I^\p H_q (X-\b X,Col(\b X)-\b X)$.

We recall now the definition of the intersection homology chain complex and groups, as in  \cite{GM1}. 
Let $T$ be an (admissible)  triangulation of $X$ such $\b X$ is triangulated by a subcomplex $L=\b T$ of $T$. Let $C^T(X)=C(T)$  denote the chain complex of simplicial chains of $X$ with respect to $T$. Let $C(X)$   denote the direct limit chain complexes under refinement of the $C^T(X)$  over all triangulations of $X$ compatible with the pl-structure. Since $\b X$ is  pl-subspace of $X$, $C^T(\b X)=C(L)$ is  defined, and is a sub complex of $C^T(X)$, and  the relative chain complex is also defined $C^T(X,\b X)=C(T,L)=C(T)/C(L)$. The construction commutes with direct limit, and hence the  $C(X)$ and $C(X,\b X)$  are defined. The {\it intersection chain group} of perversity $\p$,  is the subgroup $I^\p C_q(X)$ of $C_q(X)$ consisting of those chains $c$ such that $|c|$ is $(\p,q)$-allowable and 
$|\b c|$ is $(\p,q-1)$-allowable. The {\it relative intersection chain group}  of perversity $\p$,  is $I^\p C_q(X,\b X)=I^\p C_q(X)/I^p C_q(\b X)$, where $I^\p C_q(\b X)=C_q(\b X)$ for each $\p$, since $\b X$ is actually a manifold. The group $I^\p C_q(X)/I^p C_q(\b X)$ is the subgroup  of $C_q(X,\b X)$ consisting of those chains $c$ in $T$ that are not in $L$, such that $|c|$ is $(\p,q)$-allowable and 
$|\b c|$ is $(\p,q-1)$-allowable. 

Intersection cohomology is defined as the algebraic dual of intersection homology (see for example \cite{KW} 4.2.8). Poincar\'{e} duality is recovered for pseudomanifolds using intersection homology: with coefficients in a field, when $\b X=\emptyset$, there is an isomorphism \cite{GM1} 3.3
\beq\label{ea}
I \P_q:I^\p H_q(X)\to I^{\p^c} H^{m-q}(X).
\eeq

For a pseudomanifold with (smooth) boundary, the duality reads \cite{You}
\beq\label{eb}
I \P_q:I^\p H_q(X)\to I^{\p^c} H^{m-q}(X,\b X).
\eeq

\subsection{Basic sets}

In order to define intersection torsion and relative intersection torsion, we introduce some chain complexes of free modules. Let $X$ be a pseudomanifold of dimension $n$ with smooth boundary, and fixed stratification. First, we define the basic R sets as in  \cite{GM1} 3.4. Let $T$ be a triangulation of $X$ compatible with the filtration. Let $R^\p_q$ be the subcomplex of the first barycentric subdivision  $T'$ of $T$ consisting of all simplices which are $(\p, q)$-allowable. Then, $R^\p_q$ is a subcomplex of the $q$-skeleton of $T'$. It is clear that $R^\p_q$ is a subcomplex of $R^\p_{q+1}$. 
Define the complex $C^\p(X)$ by setting 
\[
C^\p_q(X)=H_q(R^\p_q,R^\p_{q-1}),
\]
and boundary defined by the homology long exact sequence of the pair $(R^\p_q,R^\p_{q-1})$. This is a free abelian group generated by finitely many chains with contractible support
. So  $C^\p_q(X)$ is in one one correspondence with the  group of  simplicial $q$-chains $c_q$ with $|c_q|\subset R^\p_q$, and $|\b c_q|\subset R^\p_{q-1}$. 
The homology of $C^\p(X)$ is canonically isomorphic to $\Im(H_q(R_q^\p)\to H_q(R^p_{q+1}))$. 
By \cite{GM1} 3.4, there is an  isomorphism $\Psi:\Im(H_q(R^\p_q)\to H_q(R^\p_{q+1}))\cong I^\p H_q( X)$. For this is stated, without proof, in \cite{GM1} 3.4, however, if in the present case we remove the boundary, we obtain the isomorphism for the pseudomanifold $X-\b X$, and it is clear that the groups at the two sides of the isomorphism are the same for $X$ and $X-\b X$, since the singular locus is disjoint from the boundary. Also note that the isomorphism is natural, that is to say is induced by the inclusion of $R^\p_q$ into $T'$. This is clear from the construction of the similar isomorphism called $\Psi$ for the basic sets $Q$ in \cite{GM1} 3.2.

Let $P^\p_q=R^\p_{q+1}\cap L$. Then, $P^\p_q$ is an R set $R_{q}^\p$ of $\b X$, and $\dim(R^\p_q)=q-1$.  Actually, $P^\p_q=L'_{(q)}$ is the $q$-skeleton of $L'$. For $\b X$ is a manifold and hence all the simplices of any triangulation of $\b X$ are allowable for any perversity. Define the chain complex $C^\p(\b X)$ as above.  Then,  the homology of $C^\p(\b X)$ is canonically isomorphic to $\Im(H_q(P_q^\p)\to H_q(P^\p_{q+1}))$, and there is a natural isomorphism $\Im(H_q(P^\p_q)\to H_q(P^\p_{q+1}))\cong I^\p H_q(\b X)=H_q(\b X)$, that is the restriction of $\Psi$.

Next, we deal with the relative chain complex. We define the complex $C^\p(X, \b X)$ by setting 
\[
C^\p_q(X,\b X)=H_q(R^\p_q\cup L',R^\p_{q-1}\cup L'),
\]
and boundary defined by the homology long exact sequence of the pair $(R^\p_q\cup L',R^\p_{q-1}\cup L')$. This is a free abelian group generated by finitely many chains with contractible support, and  is in one one correspondence with the group of  the simplicial $q$-chains $c_q$ with the interior of $|c_q|\subset R^\p_q-(R^\p_q\cap L')$, and the interior of $|\b c_q|\subset R^\p_{q-1}-(R^\p_{q-1}\cap L')$. It is possible to show that  the homology of $C^\m(X, \b X)$ is canonically isomorphic to $\Im(i_{q,q*}'':H_q(R_q^\p\cup L')\to H_q(R^\p_{q+1}\cup L'))$, and that $\Im(i_{q,q*}':H_q(R_q^\p\cup L')\to H_q(R^\p_{q+1}\cup L'))$ is isomorphic to the relative intersection homology of the pair $(X,\b X)$.

\subsection{R torsion}
\label{tor}

In order to define intersection torsion we briefly recall the definition of the torsion of a chain complex. We follows the classical definition of Milnor \cite{Mil}, but with a little change of notation. Let $R$ be a ring with the  invariant dimension property, and $M$  a finitely generated  free (left) $R$-module. Let $U$ be a subgroup of the group $R^\times$ of units of $R$, and let $K_U(R)=K_1(R)/U$ denotes the quotient of the Whitehead group of $R$ by the subgroup generated by the  classes of the elements of $U$.  Let $x=\{x_1,\dots, x_n\}$ and $y=\{y_1,\dots, y_n\}$ be two bases for $M$. We denote by $(y/x)$ the non singular $n$-square matrix over $R$ defined by the change of bases ($y_j=\sum_k(y/x)_{jk}x_k$), and we denote by $[y/x]$ the class of $(y/x)$ in the Whitehead group $ K_U (R)$. Let
$$
\xymatrix{
C:& C_m\ar[r]^{\b_m}&C_{m-1}\ar[r]^{\b_{m-1}}&\dots\ar[r]^{\b_2}&C_1\ar[r]^{\b_1}&C_0,
}
$$
be a bounded chain complex of finite length $m$ of (finite dimensional) free left $R$-modules. Denote by $Z_q=\ker (\b_q:C_q\to C_{q-1})$,  $B_q=\Im (\b_{q+1}:C_{q+1}\to C_q)$, and  $H_q(C)=Z_q/B_q$ the homology groups of $C$. Assume that all the chain modules $C_q$ have preferred bases $c_q=\{c_{q,1},\dots, c_{q,m_q}\}$, and the homology modules $H_q(C)$ are free with preferred bases $h_q$. Also assuming that the boundary modules $B_q$ are free with preferred bases or using stably free bases, we fix a set of elements $b_q=\{b_{q,1},\dots,b_{q,n_q}\}$ of $C_q$ such that $\b_q(b_q)$ is a basis for $B_{q-1}$ for each $q$ (in other words we are choosing a lift of a basis of $B_{q-1}$). Then the set $\{\b_{q+1}(b_{q+1}),h_q,b_q\}$ is a basis for $C_q$ for each $q$. The {\it Whitehead torsion} of $C$ with respect to the basis $h=\{h_q\}$ is the class
\[
\tau_{\rm W}(C;h)=\sum_{q=0}^m (-1)^q [(\b_{q+1}(b_{q+1}),h_q,b_q/c_q)],
\]
in the  Whitehead group $ K_U(R)$. The definition is well posed since it is possible to show that the torsion does  not depend on the bases $b_q$. If $K$ is a  connected finite cell complexes of dimension $m$, with universal covering $\tilde K$,  identify the fundamental group $\pi=\pi_1(K)$  with the group of the (cellular) covering transformations of $\tilde K$, the action $\pi$  makes each chain module $C_q(\tilde K;\Z)$ into a free  module over the group ring $\Z\pi$, finitely generated by the natural choice of the $q$-cells of $K$.  Denote the resulting complex of free finitely generated modules over $\pi$ with preferred basis (obtained by the lifts of the cells)  by $C(\tilde K;\Z\pi)$. If the homology modules $H_q(K;\Z\pi)$ are free with preferred basis $h_q$, the {\it Whitehead torsion} of $K$ with respect to the graded basis $h$ is the class
\[
\tau_{\rm W}(K;h)=w(\tau_{\rm W}(C(\tilde K;\Z\pi);h)),
\]
of $K_\pi(\Z\pi)$. If $\rho:\pi\to \Aut_R(M)$ is  a representation of the fundamental group in the group of the automorphisms of some free right module $M$ over some  ring with unit $R$, we form the twisted complex $C(K;M_\rho)$  of free finitely generated $R$-modules, by setting
\beq\label{ppp}
C_q(K;M_\rho)=M \otimes_\rho C_q( \tilde K;\Z\pi). 
\eeq

Fixing a basis $m$ for $M$, bases for these modules (and for cycles and boundary submodules) are given by tensoring with $m$. Assuming that the homology modules $H_q(C(K;M_\rho))$ are free with preferred graded bases $h=\{h_q\}$, then,  we define the {\it   R torsion} $\tau_{\rm R}(K;\rho,h)$ of $K$ with respect to the representation $\rho$ and the graded basis $h$ to be the class of $\tau_{\rm W}(C(K;M_\rho);h)$ in $\tilde K_1(\Z\Aut_R(M))/\rho(\pi)$. We have
\[
\tau_{\rm R}(K;\rho,h)=\sum_{q=0}^m (-1)^q [\rho(\b_{q+1}(b_{q+1}),h_q,b_q/c_q)].
\]

By the same  procedure we define the {\it  relative  R torsion} of the pair $(K,L)$,  and we write $\tau_{\rm R}((K,L);\rho,h)=\tau_{\rm W}(C((K,L);M_\rho);h)$. In particular, if $(W,g)$ is a compact connected oriented Riemannian manifold, and $\rho$ is an orthogonal representation, we use the de Rham maps of section \ref{dr} in order to fix the basis for the homology, following Ray and Singer \cite{RS}. We define the {\it absolute R torsion} of $(W,g)$ by
\[
\tau_{\rm R}((W,g);\rho)=\tau_{\it R}(W;\rho, \A_{\rm abs}(a)),
\]
where $a$ is an orthonormal graded basis for the harmonic forms. It is possible to prove that the definition does not depend on the basis $a$. {\it Relative $R$ torsion} for a manifold with boundary is defined accordingly.

Let $X$ be an $m$-pseudomanifold with smooth boundary, let $T$ be a triangulation of $X$ such that the boundary $\b X$ of $X$ is a subcomplex $\b T$ of $T$, and let $\tilde T$ be the  universal covering complex  of $T$, and $\tilde{\b T}$ the lift of $\b T$. Let $\tilde R^\p_q$ be the lifts of the basic sets $R_q^\p$ to $\tilde T$, and identify the fundamental group $\pi=\pi_1(X)$ with the group of the covering transformations of $\tilde T$. Note that covering transformations are simplicial, so if we set $C^\p_q(\tilde X)=H_q(\tilde R^\p_q,\tilde R^\p_{q-1})$,  the action of the group of covering transformations group makes each chain group $C^\p_q(\tilde X)$ into a free module over the group ring $\Z\pi$, and each of these modules is  finitely generated by fixing lifts  of the natural choice of the $q$-chains that generate $C^\p_q(X)$. We obtain a complex of free finitely generated modules over $\Z\pi$ that we denote by $C^\p(\tilde X;\Z\pi)$, with preferred basis. The same procedure applies for the relative chain complex $C^\p_q(\tilde X,\tilde{\b X})=H_q(\tilde R^\p_q\cup \tilde{ \b T'},\tilde R^\p_{q-1}\cup\tilde{\b T'})$, and gives the $\Z\pi$-complex  $C^\p((\tilde X,\b \tilde X);\Z\pi)$, with preferred basis obtained by lifting the chains whose supports  do not intersect the boundary.

Assuming that the homology modules $H_q(C^\p(\tilde X;\Z\pi))$, $H_q(C^\p((\tilde X,\b \tilde X);\Z\pi))$ are $\Z\pi$-free with preferred graded bases $h=\{h_q\}$, we define the {\it intersection Whitehead  torsion} of $X$ and the {\it relative intersection Whitehead  torsion} of the pair $(X,\b X)$ with respect to the graded basis $h$ to be the classes 
\[
I\tau^\p_{\rm W}(X;h)=\tau_{\rm W}(C^\p(\tilde X;\Z\pi);h),\hspace{30pt} I\tau^\p_{\rm W}((X,\b X);h)=\tau_{\rm W}(C^\p((\tilde X,\b \tilde X);\Z\pi);h),
\] 
in the Whitehead group $Wh(\pi_1(X))= K_\pi(\Z\pi)$, respectively.  Proceeding as in the smooth case, given a representation $\rho$ of $\pi_1(X)$ we define   {\it  intersection R torsion} $I^\p\tau_{\rm R}(X;\rho,h)= \tau_{\rm W}(C^\p(X;M_\rho);h)$ of $X$ with respect to the representation $\rho$ and to the graded basis $h$, and  the {\it  relative intersection R torsion} of the pair $(X,\b X)$,  $I^\p \tau_{\rm R}((X,\b X);\rho,h)=\tau_{\rm W}(C^\p((X,\b X);M_\rho);h)$. If in particular a Riemannian structure is defined on the non singular part of $X$, $L^2$ forms can be used to extend the construction of Ray and Singer and to define suitable de Rham maps from $L^2$ harmonic forms to intersection homology, and to fix the basis $h$ \cite[Section 2, pg. 197]{Dar1} (and locally cited \cite{Che3}). Note in particular, that the basis $h$ fixed in this way is self dual, i.e. $I\P_q(h_q)$ is the algebraic dual of $h_{n-q}$. Following A. Dar \cite[pg. 197]{Dar1}, we use the notation $I^\p \tau_{\rm R}((X,g);\rho)$ and $I^\p \tau_{\rm R}((X,\b X,g);\rho)$ to denote the torsion $I^\p \tau_{\rm R}((X,\b X);\rho,h)$ when the basis $h$ is fixed as the image of an orthonormal basis of $L^2$ harmonic forms via de De Rham map, and we define the {\it intersection R  torsion} of $X$, and the {\it relative intersection R torsion} of $(X,\b X)$ with respect to the representation $\rho$ by
\begin{align*}
I \tau_{\rm R}((X,g);\rho)&=\frac{1}{2}\left(I^\m \tau_{\rm R}((X,g);\rho)+I^{\m^c} \tau_{\rm R}((X,g);\rho)\right),\\
I \tau_{\rm R}((X,\b X,g);\rho)&=\frac{1}{2}\left(I^\m \tau_{\rm R}((X,\b X,g);\rho)+I^{\m^c} \tau_{\rm R}((X,\b X,g);\rho)\right).
\end{align*}

In all definitions, if  $X$ is an oriented manifold stratified with only one stratum $X$ then  we obtain the classical Whitehead torsion  and the classical R torsion.


\section{Duality theorems for intersection R torsion of pseudomanifolds with smooth boundary}

We give in this section some duality theorems for the intersection torsion of a pseudomanifold with smooth boundary that extend the duality theorems of  A. Dar  \cite{Dar1} for the boundary less case. First, we need a lemma for the $L^2$ harmonics form on the suspension and on the cone. Due to the independent interest of these results, we give them also for the case of an even dimensional section. We present an explicit proof for the suspension. The proof for the cone is similar, and in the odd dimensional case was given in Lemma 3.5 of \cite{HS2}.    Next, we give some formulas for the torsion, that we use to prove the final duality results.

Let $\Sigma_l W=(0,2l)\times W\cup \{p_0,p_{2l}\}$ be the suspension of $W$, realized as the  gluing of two copies of $C_l W$ along the boundaries.

\begin{lem}\label{llI} If  $\dim W=2p-1$ is odd,  then ($\alpha_q=\frac{1}{2}(1+2q-2p)$)
\[
\H^q(\Sigma_l W) = \begin{cases} \H^q(W), & 0\leq q\leq p-1,\\
 \{0\},& q = p,\\
\left\{x^{2\alpha_q-1}dx \wedge \varphi^{(q-1)}, \varphi^{(q-1)}\in \H^{q-1}(W)\right\}, & p+1\leq q\leq 2p.
\end{cases}
\]

If $\dim W=2p$ is even, then
\[
\H^q(\Sigma_l W) = \begin{cases} \H^q(W), & 0\leq q\leq p,\\
 \{0\},& q = p+1,\\
\H^{q-1}( W), & p+2\leq q\leq 2p+1.
\end{cases}
\]

\end{lem}
\begin{proof} The solutions of the harmonic equation $\Delta u=0$ on $\Sigma_l W$ are 
\[
u(x)=\begin{cases} f_1(x),& 0\leq x\leq l,\\ f_2(2l-x), &l\leq x<2l,\end{cases}
\]
where $f_1$ and $f_2$ are forms of the following four types:
\begin{align*}
\psi^{(q)}_{\pm, 1,n} =& x^{a_{\pm,q,n}} \varphi_{{\rm ccl},n}^{(q)},\\
\psi^{(q)}_{\pm, 2,n} =& x^{a_{\pm,q-1,n}} \tilde d\varphi_{{\rm ccl},n}^{(q-1)}+a_{\pm,q-1,n}x^{a_{\pm,q-1,n}-1}dx\wedge \varphi_{{\rm ccl},n}^{(q-1)},\\
\psi^{(q)}_{\pm,3,n} =& x^{a_{\pm,q-1,n}+2} \tilde d\varphi_{{\rm ccl},n}^{(q-1)}+a_{\mp,q-1,n}x^{a_{\pm,q-1,n}+1}dx\wedge \varphi_{{\rm ccl},n}^{(q-1)},\\
\psi^{(q)}_{\pm,4,n} =& x^{a_{\pm,q-2,n}+1}dx\wedge \tilde d\varphi_{{\rm ccl},n}^{(q-2)}.
\end{align*}
and $a_{\pm, q,n}=\alpha_q\pm \mu_{q,n}$. The harmonics of $\Delta$ are obtained requiring that $u$, $ du$ and $d^\dagger$ are square integrable, and  satisfy the following conditions
\[
\begin{cases} \left. f_1(x)\right|_{x=l}=\left. f_2(x)\right|_{x=l},\\
\left. (d f_1)(x)\right|_{x=l}=\left. (d f_2)(x)\right|_{x=l},\\
\left. (d^\dagger f_1)(x)\right|_{x=l}=\left. (d^\dagger f_2)(x)\right|_{x=l},
\end{cases}
\]
plus the ideal boundary condition of Cheeger if $\dim W=2p$ is even \cite{Che2,Che3}, described below. 

Note that the first condition is always satisfied, while the other two conditions coincide respectively with the conditions
$\left. (d f_1)_{\rm norm}(x)\right|_{x=l}=0$ and $\left. (d^\dagger f_1)_{\rm tg}(x)\right|_{x=l}=0$. An explicit verification of these conditions gives the result. 

Let us  consider the case of $\dim W=2p-1$ odd in some details. The  solution of type I $\psi^{(q)}_{+, 1, n}$ satisfies the square integrability condition  for all $q$, while  $\psi^{(q)}_{-, 1, n}$ satisfies this condition only if $q=p-1$. The condition $\left. (d f_1)_{\rm norm}(x)\right|_{x=l}=0$ is satisfied if and only if $a_{\pm, q,n}=0$, and this is true for $a_{+, q,n}$ when $\lambda_{q,n}=0$ and $0\leq q\leq p-1$, and for $a_{-, q,n}$ when  $\lambda_{q,n}=0$ and $p\leq q\leq 2p-1$. Since $\lambda_{q,n}=0$, $\varphi_{{\rm ccl},n}^{(q)}$ is an harmonic on $W$, and thus $d^\dagger f_1=0$.

The solution of type II $\psi^{(q)}_{+, 2, n}$ satisfies the square integrability condition  for all $q$, while  $\psi^{(q)}_{-, 2, n}$ satisfies this condition only if $q=p$. The condition $\left. (d^\dagger f_2)_{\rm norm}(x)\right|_{x=l}=0$ is satisfied if and only if either $a_{\pm, q,n}=0$, or $\varphi_{{\rm ccl},n}^{(q-1)}$ is an harmonic $h$ on $W$. The first possibility would require $\lambda_{q-1,n}=0$, and therefore would give $\psi^{(q)}_{\pm, 2, n} =0$. The second possibility gives $\psi^{(q)}_{\pm, 2, n}=a_{\pm,q-1,n}x^{a_{\pm,q-1,n}-1}dx\wedge h$. Applying $\left. (d^\dagger f_1)_{\rm tg}(x)\right|_{x=l}=0$, we obtain 
$p\leq q\leq 2p-1$ for $\psi^{(q)}_{+, 2, n}$, and $0\leq q\leq p$  for $\psi^{(q)}_{-, 2, n}$,  where however, the case $q=p$ corresponds to the null form. 

Proceeding in a similar way, we find that there are no new forms  of the types III and IV satisfying the conditions above.

If  $\dim W=2p$ is even, the same analysis gives the answer in all dimension $q\not=p+1$. In dimension $q=p+1$, we obtain the form $\varphi=dx\wedge h$, where $h$ is an harmonic  on $W$, that satisfies all the conditions above and is square integrable with its exterior derivative, thus we need the ideal boundary conditions, that we briefly recall here. Let $H^p(W;\Q)=V_a\oplus V_r$ be  a maximal self annihilating (for the cup product) decomposition. Let $\{e_j\}_{j=1}^h$ be an orthonormal basis for the corresponding spaces of harmonic $p$-forms $\H^p(W)$ coherent with the decomposition. Then, for any $h\in \H^p(W)$,
\[
h=\sum_{j=1}^p h_{a,j}(x) e_j +\sum_{j=p+1}^{2p} h_{r,j}(x) e_j.
\]

We say that a form 
\[
\omega=\alpha+dx\wedge\beta
\]
in $\H^p(\Sigma_l W)$ satisfies the  ideal boundary condition (is in the domain of $d$)  for the decomposition $V_a\oplus V_r$ if for the  decomposition in $V_a\oplus V_r$ of the projection of $\alpha$ onto $\H^p(W)$ the following condition holds:
\begin{align*}
f_{a,j}'(0)=f_{r,j}(0)=0.
\end{align*}

Similarly, a $p+1$-form $\omega$ is in the domain of $d^\dagger$ for the decomposition $V_a\oplus V_r$ if $\star \omega$ is in the domain of $d$ for the decomposition  $\star V_r\oplus \star V_a$. Consider the $p+1$-form $\varphi=dx\wedge h$. Its dual is $\star \varphi=\tilde \star h \in \H^p(\Sigma_l W)$, and clearly does not satisfy the ideal BC.  \end{proof}


\begin{lem}\label{ll} If  $\dim W=2p-1$ is odd,  then ($\alpha_q=\frac{1}{2}(1+2q-2p)$)
\begin{align*}
\H^q_{\rm abs}(C_l W)&=\begin{cases}\H^q(W), &0\leq q\leq p-1,\\
 \{0\}, & p\leq q\leq 2p-1.\end{cases}\\
\H^q_{\rm rel}(C_l W)&=\begin{cases} \{0\}, & \hspace{17pt}0\leq q\leq p,\\
\left\{x^{2\alpha_q-1}dx \wedge \varphi^{(q-1)}, \varphi^{(q-1)}\in \H^{q-1}(W)\right\}, &p+1\leq  q\leq 2p.
\end{cases}
\end{align*}

If $\dim W=2p$ is even, then
\begin{align*}
\H^q_{\rm abs}(C_l W)&=\begin{cases}\H^q(W), &0\leq q\leq p,\\
 \{0\}, & p+1\leq q\leq 2p+1;\end{cases}\\
\H^q_{\rm rel}(C_l W)&=\begin{cases} \{0\}, & \hspace{17pt}0\leq q\leq p,\\
\left\{x^{2\alpha_{q-1}-1}dx \wedge \varphi^{(q-1)}, \varphi^{(q-1)}\in \H^{q-1}(W)\right\}, &p+1\leq  q\leq 2p+1.
\end{cases}
\end{align*}

\end{lem}

For the suspension,  we have a short exact sequence of chain complexes:
\beq\label{ee1}
0\to C^\p (\b C_lW)\to C^\p (C_l W)\oplus C^\p (C_lW)\to C^\p(\Sigma_l W)\to 0.
\eeq

A formula for the torsion of an exact sequence of complexes is given by  Milnor in \cite{Mil} Section 3. In the present case, we can fix the chain basis of the middle  complex consistently, using the   basis determined by the simplices, and hence we have the following formula
\[
2\log I^\p \tau_{\rm R}((C_l W,g_C);\rho_0)=\log \tau_{\rm R}((W,l^2 g);\rho_0)+\log I^\p \tau_{\rm R}((\Sigma_l W,g_\Sigma);\rho_0)+\log \tau(\S^\p_m),
\]
where the complex $\S^\p_m$  is defined by the exact long homology sequence associated to the exact sequence in equation (\ref{ee1}), that is the Mayer Vietoris sequence
\[
\S^\p_m:{\xymatrix@C=0.5cm{& \ldots \ar[r] &I^{\p}H_q(\b C_lW)\ar[r] & \ar[r] I^{\p}H_q(C_l W) \oplus I^{\p}H_q(C_l W) & \ar[r] I^{\p}H_q(\Sigma_l W)& \ldots}},
\]
i.e., $\S^\p_{m,3q} = I^{\p} H_q (W)$, $\S^\p_{m,3q+1} = I^{\p} H_q (C_l W)\oplus I^{\p}H_q(C_l W)$ and $\S^\p_{m,3q+2} =I^\p H_q (\Sigma_l W)$.

\begin{lem} \label{l1} Let $W$ be a compact connected oriented manifold of odd dimension $m=2p-1$ without boundary. 
Let $\rho_0:\pi_1(C_l W)\to O(1,\R)$ be the rank one trivial orthogonal representation of the fundamental group. Then, 
\[
2\log I^\m \tau_{\rm R}((C_l W,g_C);\rho_0)=\log \tau_{\rm R}((\b C_lW,l^2 g);\rho_0)+\log I^\m \tau_{\rm R}((\Sigma_l W,g_\Sigma);\rho_0)+\log \tau(\S_m^\m),
\]
where $r_q = {\rm rk} H_q(W)$, and
\begin{align*}
\log \tau(\S_{2p-1}^\m)=\log\tau(\S_{2p-1}^{\m^c})&=\sum_{q=0}^{p-1} (-1)^{q}  \log\left(\frac{l}{2p-2q}\right)^{r_{q}}.
\end{align*}
\end{lem}
\begin{proof} 





First, we  recall the intersection homology for the cone and the suspension with middle perversity. Since both spaces have isolated singularities, the unique value which imports of the perversity is the value  $\m_{2p} = \m^c_{2p}=p-1$. We have (see for example \cite{GM1} 6, or \cite{KW} 4.7.2, 4.7.3)
\[
I^{\m}H_q(C_l W) =I^{\m^c}H_q(C_l W) = \left\{\begin{array}{ll} H_q(\b C_lW), & q<p,\\
0, & q\geq p,
\end{array}\right.
\]
and
\[
I^{\m}H_q(\Sigma_l W) =I^{\m^c}H_q(\Sigma_l W) = \left\{\begin{array}{ll} H_q(\b C_lW), & q<p,\\
\Im\left(H_q(W) \to H_q(\Sigma_l W)\right) = 0,& q = p,\\
H_q(\Sigma_l W), & q>p,
\end{array}\right.
\]

Note that, beside the homology, also the basic $R$ set with the two middle complementary perversities   coincide, by the very definition. This implies that the chain complex used in the definition of the intersection torsion for these two perversities coincide and therefore the torsions coincide. We proceed by taking $\p=\m$, and this will also cover the complementary case.

Next,  in order to compute the torsion of the complex $\S^\m_m$, we need the chain bases. These are the bases for the homology determined by the geometry using the de Rham maps as described in Sections \ref{dr} and \ref{tor}. 
We study  the two cases $q<p$ and $q\geq p$ separately. When $q<p$, consider the following part of the complex $\S^\m_{2p-1}$
\[
{\xymatrix@C=0.5cm{\ar[r]\ldots & I^{\m} H_{q+1}(\Sigma_l W) \ar[r]^{\b_{q+1}} & I^{\m}H_q(\b C_lW)\ar[r]& \ar[r] I^{\m}H_q(C_l W) \oplus I^{\m}H_q(C_l W) & \ar[r]^{\hspace{20pt}\b_q} I^{\m}H_q(\Sigma_l W)& \ldots}}.
\] 

The geometry implies that  the homomorphisms $\b_{q+1} = \b_q$ are null, and using the previous results, all the vector spaces are isomorphic to $V=H_q(W)$, and the sequence splits as
\beq\label{s1a}
{\xymatrix@C=0.3cm{0 \ar[r] &  H_q(\b C_lW )\cong V\ar[r] & \ar[r] I^\m  H_q(C_l W)\cong V \oplus I^\m H_q(C_l W)\cong V & \ar[r] I^\m H_q(\Sigma_l W)\cong V& 0}}.
\eeq

In order to fix the homology bases, let $a_q$ be an  orthonormal base for $\H^{q}(W,g)$. Then the norm of $a_{q,j}$ with the metric $l^2  g$ is
\[
|| a_{q,j} ||_{l^2  g}^2 = \int_W a_{q,j} \wedge \star_{l^2  g} a_{q,j}  = l^{2p-1-2q} \int_W  a_{q,j} \wedge \star_{g} a_{q,j} = l^{2p-1-2q} ||a_{q,j}||^2_{ g} = l^{2p-1-2q}.
\]

So a orthonormal base for $\H^{q}(W,l^2  g)$ is $l^{-\frac{2p-1-2q}{2} }a_q$, and applying the de Rham maps we obtain
\begin{align*}
\A_{q,l^2  g}(l^{-\frac{2p-1-2q}{2}}a_{q,j}) &= l^{-\frac{2p-1-2q}{2}} \A_{q,l^2  g}(a_{q,j}) = l^{-\frac{2p-1-2q}{2}} \P_q^{-1} \A^{2p-1-q} \star_{l^2  g} (a_{q,j})\\
&= l^{-\frac{2p-1-2q}{2}} l^{2p-1-2q}  \P_q^{-1} \A^{2p-1-q} \star_g (a_{q,j}) = l^{\frac{2p-1-2q}{2}} \A_{q, g}(a_{q,j}).
\end{align*}

Then the basis for $H_q(\b C_l W)$ is $l^{\frac{2p-1-2q}{2}} \A_{q, g}(a_q)$. Next, consider the cone $(C_l W,g_C)$. By Lemma \ref{ll}, the constant extension of the forms in $a_q$ gives a basis for $\H^{q}_{\rm abs}(C_l W)$. The norm of this basis elements is
\[
|| a_{q,j} ||_{g_C}^2 = \int_{C_lW} a_{q,j} \wedge \star_{g_C} a_{q,j}  = \int_{0}^{l} x^{2p-1-2q} dx \int_W  a_{q,j} \wedge \star_{ g} a_{q,j} = \frac{l^{2p-2q}}{2p-2q} ||a_{q,j}||^2_{ g} = \frac{l^{2p-2q}}{2p-2q}.
\] 

So an orthonormal base for $\H^{q}(C_l W)$ is $\left(\frac{l^{2p-2q}}{2p-2q}\right)^{-\frac{1}{2}} a_q$, and, using duality (\ref{eb}),
\begin{align*}
\A^{\rm abs}_{q, g_C}\left(\left(\frac{l^{2p-2q}}{2p-2q}\right)^{-\frac{1}{2}}a_{q,j}\right) &= \left(\frac{l^{2p-2q}}{2p-2q}\right)^{-\frac{1}{2}} \A^{\rm abs}_{q,g_C}(a_{q,j}) = \left(\frac{l^{2p-2q}}{2p-2q}\right)^{-\frac{1}{2}} I \P_q^{-1} \A^{2p-q}_{\rm rel} \star_{g_C} (a_{q,j})\\
&= \left(\frac{l^{2p-2q}}{2p-2q}\right)^{-\frac{1}{2}} \left(\frac{l^{2p-2q}}{2p-2q}\right) \P_q^{-1} \A^{2p-1-q} \star_{ g} (a_{q,j}) \\
&=\left(\frac{l^{2p-2q}}{2p-2q}\right)^{\frac{1}{2}}\A_{q, g}(a_{q,j}).
\end{align*}

This give the basis for  $I^\m H_q(C_l W, g_C)$: $\left(\frac{l^{2p-2q}}{2p-2q}\right)^{\frac{1}{2}} \A_{q, g}(a_{q})$. Repeating the same process for $\H^{q}(\Sigma_l W)$ we obtain the basis of $I^\m H_q(\Sigma_l W)$: $\left(\frac{l^{2p-2q}}{p-q}\right)^{\frac{1}{2}} \A_{q, g}(a_{q})$. We can now compute the determinants of the change of basis in the vector spaces of the sequence in equation (\ref{s1a}). At
 $I^\m H_q(\b C_lW)$ the determinant is $1$, at $I^\m H_q(C_l W) \oplus H_q(C_l W)$ is $\left(\frac{l}{2p-2q}\right)^{-\frac{r_q}{2}}$ and at $I^\m H_q(\Sigma_l W)$ is $2^{-\frac{r_q}{2}}$. We consider now the case $q\geq p$. The relevant part of the sequence  $\S^\m_{2p-1}$ reads
\[
{\xymatrix@C=0.5cm{\cdots\ar[r]& 0  \ar[r] & I^\m  H_{q+1}(\Sigma_l W)  \ar[r]^{\b_{q+1}} & H_q(\b C_l W)\ar[r]  &  0 \ar[r]&\cdots }}
\] 

Since the de Rham maps are self dual, as observed at the end of Section \ref{tor}, and the intersection homology with perversities $\m$ and $\m^c$ coincide, we can use duality (\ref{ea}) to obtain the basis for  $I^{\m^c} H_{q+1}(\Sigma_l W)$ starting with the basis for the same space obtained when $q<p$. 
This gives the basis $\left(\frac{l^{2q-2p+2}}{q+1-p}\right)^{-\frac{1}{2}} \A_{q, g}(a_{q})$ for $I^{\m^c} H_{q+1}(\Sigma_l W)$. Using the basis fixed above for the other space, the determinant of the change of basis at $H_{q}(\b C_l W)$ is $1$ and at $I^\m H_{q+1}(\Sigma_l W)$ is $\left(\frac{l}{q+1-p}\right)^{\frac{r_q}{2}}$. Now applying the definition of Reidemeister torsion to the complex $\S^\m_{2p-1}$, we obtain (where $D$ denotes the determinant of the matrix of the change of basis)
\begin{align*}
\log\tau(\S^\m_{2p-1}) =& \sum_{q=0}^{6p} (-1)^q \log D(\S^\m_{2p-1,q}) \\
=&\sum_{q=0}^{2p} (-1)^{q} \log D(I^{\m}H_{q}(\Sigma_l W)) + \sum_{q=0}^{p-1} (-1)^{q+1}\log D( I^{\m}H_{q}(C_l W)\oplus I^{\m}H_q(C_l W))\\
=&\sum_{q=0}^{p-1} (-1)^{q} \log D( I^{\m}H_{q}(\Sigma_l W)) + \sum_{q=p+1}^{2p} (-1)^{q} \log D( I^{\m}H_{q}(\Sigma_l W))\\
&+ \sum_{q=0}^{p-1} (-1)^{q+1}\log D (I^{\m}H_{q}(C_l W)\oplus I^{\m}H_q(C_l W)) \\
=&\sum_{q=0}^{p-1} (-1)^{q} \log 2^{-\frac{r_q}{2}} + \sum_{q=0}^{p-1} (-1)^{q+1}\log \left(\frac{l}{2p-2q}\right)^{-\frac{r_q}{2}} + \sum_{q=p+1}^{2p} (-1)^{q}  \log\left(\frac{l}{q-p}\right)^{\frac{r_{q-1}}{2}}\\
=&\sum_{q=0}^{p-1} (-1)^{q} \log 2^{-\frac{r_q}{2}} + \sum_{q=0}^{p-1} (-1)^{q+1}\log \left(\frac{l}{2p-2q}\right)^{-\frac{r_q}{2}} + \sum_{q=0}^{p-1} (-1)^{q}  \log\left(\frac{l}{p-q}\right)^{\frac{r_{q}}{2}}\\
=&\sum_{q=0}^{p-1} (-1)^{q}  \log\left(\frac{l}{2p-2q}\right)^{r_{q}},
\end{align*}
and this completes the proof.  \end{proof}

Considering the short exact sequence of chain complexes associated to the pair $(C_l W, \b C_l W)$,
\[
0\to C^\p(\b C_l W)\to C^\p (C_l W)\to C^\p(C_l W,\b C_l W)\to 0,
\]
by Milnor \cite{Mil} 3, we have  
\[
\log I^\p \tau_{\rm R}((C_l W,g_C);\rho_0)=\log \tau_{\rm R}((\b C_lW,l^2 g);\rho_0)+\log I^\p \tau_{\rm R}((C_l W,\b C_l W),g_C);\rho_0)+\log \tau(\T^\p_m).
\]

The calculation of the  torsion of 
\[
\T^\p_m:{\xymatrix@C=0.5cm{& \ldots \ar[r] &I^{\p}H_q(\b C_lW)\ar[r] & \ar[r] I^{\p}H_q(C_l W) & \ar[r] I^{\p}H_q(C_l W,\b C_l W)& \ldots}},
\]
where $\T^\p_{m,3q} = I^{\p} H_q (W)$, $\T^\p_{m,3q+1} = I^{\p} H_q (C_l W)$ and $\T^\p_{m,3q+2} =I^\p H_q (C_l W,\b C_l W)$,
will give the following result.

\begin{lem} \label{l2} Let $W$ be a compact connected oriented manifold of odd dimension $m=2p-1$ without boundary. 
Let $\rho_0:\pi_1(C_l W)\to O(1,\R)$ be the rank one trivial orthogonal representation of the fundamental group. Then, 
\[
\log I^\m \tau_{\rm R}((C_l W,g_C);\rho_0)=\log \tau_{\rm R}((\b C_lW,l^2 g);\rho_0)+\log I^\m \tau_{\rm R}((C_l W,\b C_l W),g_C);\rho_0)+\log \tau(\T^\m_m).
\]
where $r_q = {\rm rk} H_q(W)$, and
\begin{align*}
\log \tau(\T_{2p-1}^\m)=\log\tau(\T_{2p-1}^{\m^c})&=\sum_{q=0}^{p-1} (-1)^{q}  \log\left(\frac{l}{2p-2q}\right)^{r_{q}}.
\end{align*}
\end{lem}
\begin{proof} 

The intersection homology of $C_l W$ was given in the proof of Lemma \ref{l1}, and that of $(C_l W, \b C_l W)$ can be computed using the sequence of the pair, we get
\[
I^{\m}H_q(C_l W,\b C_l W) =I^{\m^c}H_q(C_l W,\b C_l W) = \left\{\begin{array}{ll} 0, & q\leq p,\\
H_{q-1}(\b C_lW), & q> p.
\end{array}\right.
\]

 Torsion and homology for the complementary perversities $\m$ and $\m^c$ coincide, so fix $\p=\m$. When $q<p$, consider the following part of the complex $\T^\m_{2p-1}$
\[
{\xymatrix@C=0.5cm{& I^\m H_q(C_l W,\b C_l W)=0 \ar[r] &I^{\m}H_q(\b C_lW)\ar[r] & \ar[r] I^{\m}H_q(C_l W) & I^{\m}H_q(C_l W,\b C_l W)=0}}.
\]

Let $a_q$ be an  orthonormal base for $\H^{q}(W)$. Then, as in the proof of Lemma \ref{l1}, a basis for  $H_q(\b C_l W)$ is $l^{\frac{2p-1-2q}{2}} \A_{q, g}(a_q)$, and a basis for  $I^\m H_q(C_l W,\b C_l W)$ is $\left(\frac{l^{2p-2q}}{2p-2q}\right)^{\frac{1}{2}} \A_{q, g}(a_{q})$. The determinant of the change of basis is 1 at 
 $I^\m H_q(\b C_lW)$, and is  $\left(\frac{l}{2p-2q}\right)^{-\frac{r_q}{2}}$  at $I^\m H_q(C_l W)$. When  $q\geq p$,  the relevant part of the sequence  $\T^\m_{2p-1}$ is
\beq\label{s1b}
{\xymatrix@C=0.5cm{ 0  \ar[r] & I^\m  H_{q+1}(C_l W, \b C_l W)  \ar[r] & H_q(\b C_l W)\ar[r]  &  0  }}
\eeq

By Lemma \ref{ll}, a basis for harmonic forms with relative boundary conditions $\H^q_{\rm rel}(C_l W)$ is $\omega_q=x^{2\alpha_{q-1}-1} dx \wedge a_{q-1}$. Their norm is
\[
|| \omega_{q,j} ||_{g_C}^2 
= \int_{C_lW} x^{2q-2p-1} dx \wedge a_{q-1,j}\wedge \star_{g} a_{q-1,j}
= \int_{0}^{l} x^{2q-2p-1} dx  ||a_{q-1,j}||^2_{ g} = \frac{l^{2q-2p}}{2q-2p}.
\] 

So an orthonormal base for $\H^{q}_{\rm rel}(C_l W)$ is $\left(\frac{l^{2q-2p}}{2q-2p}\right)^{-\frac{1}{2}} \omega_q$, using duality (\ref{eb}),
\begin{align*}
\A^{\rm rel}_{q, g_C}\left(\left(\frac{l^{2q-2p}}{2q-2p}\right)^{-\frac{1}{2}}\omega_{q,j}\right) &= \left(\frac{l^{2q-2p}}{2q-2p}\right)^{-\frac{1}{2}} \A^{\rm rel}_{q,g_C}(\omega_{q,j}) = \left(\frac{l^{2q-2p}}{2q-2p}\right)^{-\frac{1}{2}} I \P_q^{-1} \A_{\rm abs}^{2p-q} \star_{g_C} (\omega_{q,j})\\
&= \left(\frac{l^{2q-2p}}{2q-2p}\right)^{-\frac{1}{2}} \left(\frac{l^{2q-2p}}{2q-2p}\right) \P_{q-1}^{-1} \A^{2p-1-(q-1)} \star_{ g} (a_{q-1,j}) \\
&=\left(\frac{l^{2q-2p}}{2q-2p}\right)^{\frac{1}{2}}\A_{q-1, g}(a_{q-1,j}),
\end{align*} 
and this gives the basis for $I^\m H_q(C_l W,\b C_l W)$. The determinants of the change of basis in (\ref{s1b}) are: $1$ at 
$H_q(\b C_l W)$, and $\left(\frac{l}{2q-2p+2}\right)^{\frac{r_q}{2}}$ at $I^\m H_{q+1}(C_l W,\b C_l W)$. Applying the definition of Reidemeister torsion to the complex $\T^\m_{2p-1}$, we obtain

\begin{align*}
\log\tau(\T^\m_{2p-1}) =& \sum_{q=0}^{6p} (-1)^q \log D(\T^\m_{2p-1}) \\
=&\sum_{q=p+1}^{2p} (-1)^{q} \log D(I^{\m}H_{q}(C_l W,\b C_l W)) + \sum_{q=0}^{p-1} (-1)^{q+1}\log D( I^{\m}H_{q}(C_l W))\\
=&\sum_{q=p+1}^{2p} (-1)^{q}  \log  \left(\frac{l}{2p-2q}\right)^{\frac{r_{q-1}}{2}} + \sum_{q=0}^{p-1} (-1)^{q+1}  \log\left(\frac{l}{2q-2p}\right)^{\frac{-r_{q}}{2}}\\
=&\sum_{q=0}^{p-1} (-1)^{q}  \log\left(\frac{l}{2p-2q}\right)^{r_{q}},
\end{align*}
and this complete the proof. \end{proof}


\begin{prop} \label{p1} Let $W$ be a compact connected oriented manifold of odd dimension $m=2p-1$ without boundary. 
Let $\rho_0:\pi_1(C_l W)\to O(1,\R)$ be the rank one trivial orthogonal representation of the fundamental group. Then ($r_q = {\rm rk} H_q(W)$), 
\begin{align*}
\log I \tau_{\rm R}((C_l W,g_C);\rho_0)=&-\log I \tau_{\rm R}((C_l W,\b C_l W,g_C);\rho_0)\\
=&\frac{1}{2}\log \tau_{\rm R}((\b C_lW,l^2 g);\rho_0)+\frac{1}{2}\sum_{q=0}^{p-1} (-1)^{q}  \log\left(\frac{l}{2p-2q}\right)^{r_{q}}.
\end{align*}
\end{prop}
\begin{proof} By  definition, if $m$ is odd, 
\begin{align*}
\log I \tau_{\rm R}((C_l W,\b C_l W,g_C);\rho_0)=&\log I^{\m} \tau_{\rm R}((C_l W,\b C_l W,g_C);\rho_0)
=\log I^{\m^c} \tau_{\rm R}((C_l W,\b C_l W,g_C);\rho_0).
\end{align*}

Then the statement follows using Lemma \ref{l1} once we recall that, when $m$ is odd,  by  \cite{Dar1} 2.8, 
\[
\log I^{\m^c} \tau_{\rm R}((\Sigma_l W,g_\Sigma);\rho_0)=-\log I^\m \tau_{\rm R}((\Sigma_l W, g_\Sigma);\rho_0).
\]
\end{proof}

\begin{theo}  Let $W$ be a compact connected oriented manifold of odd dimension $m=2p-1$ without boundary. 
Let $\rho_0:\pi_1(C_l W)\to O(1,\R)$ be the trivial orthogonal representation of the fundamental group. Then, 
\begin{align*}
\log I^\m \tau_{\rm R}((C_l W,g_C);\rho_0)&=(-1)^m\log I^{\m^c} \tau_{\rm R}((C_l W,\b C_l W,g_C);\rho_0),\\
\log I \tau_{\rm R}((C_l W,g_C);\rho_0)&=(-1)^{m}\log I \tau_{\rm R}((C_l W,\b C_l W,g_C);\rho_0).
\end{align*}
\end{theo}
\begin{proof} The proof  follows by Proposition \ref{p1} and the two previous lemmas. \end{proof}

\vskip .2in



\end{document}